\documentclass[11pt]{amsart}
\usepackage{amsmath,amsthm,amsfonts,amssymb,amscd}
\usepackage[latin1]{inputenc}
\usepackage{palatino}
\usepackage{enumerate}
\usepackage{graphicx}

\newtheorem{Theorem}{Theorem}[section]
\newtheorem{Corollary}[Theorem]{Corollary}

\newtheorem{Lemma}[Theorem]{Lemma}
\newtheorem{Claim}[Theorem]{Claim}
\theoremstyle{Definition}

\newtheorem{Example}[Theorem]{Example}

\theoremstyle{Remark}
\newtheorem{Remark}[Theorem]{Remark}

\def\vr{{\varphi}}

\newcommand{\dint}{\displaystyle\int}

\def\Re{\operatorname{{Re}}}

\def\sing{\operatorname{{sing}}}

\def\codim{\operatorname{{codim}}}

\title[On integral conditions for the existence of first integrals]{On integral conditions for the existence of first integrals analytic saddle singularities}
\author{V. Le\'on and B. Sc\'ardua}
\address{V. Le\'on. ILACVN - CICN, Universidade Federal da Integração Latino-Americana, Parque tecnológico de Itaipu, Foz do Iguaçu-PR, 85867-970 - Brazil}
\email{victor.leon@unila.edu.br}
\address{B. Sc\'ardua. Instituto de Matem\'atica - Universidade Federal do Rio de Janeiro,
CP. 68530-Rio de Janeiro-RJ, 21945-970 - Brazil}
\email{bruno.scardua@gmail.com}

\date{}
\begin{document}

\maketitle

\begin{abstract}
We study one-parameter analytic integrable deformations of the germ of $2\times(n-2)$-type complex saddle  singularity given by  $d(xy)=0$ at the origin $0 \in \mathbb C^2\times \mathbb C^{n-2}$. Such a deformation writes ${\omega}^t=d(xy) + \sum\limits_{j=1}^\infty t^j \omega_j$
where $t\in \mathbb C,0$ is the parameter of the deformation and the coefficients  $\omega_j$ are holomorphic one-forms in some neighborhood of the origin $0\in \mathbb C^n$. We prove that, under a nondegeneracy condition of the singular set of the deformation, with respect to the
fibration $d(xy)=0$, the existence of a holomorphic first integral for each element ${\omega}^t$ of the deformation is equivalent to the vanishing of certain line integrals $\oint_{\gamma_c}{\omega}^t=0, \forall \gamma_c, \forall t$ calculated on
cycles $\gamma_c$ contained in the fibers $xy=c, \,0 \ne  c \in \mathbb C,0$. This result is quite sharp regarding the conditions of the singular set and on the vanishing of the integrals in cycles.
It is also  not valid for ramified saddles, i.e., for deformations of saddles of the form $x^ny^m=c$ where $n+m>2$.  As an application of our techniques we obtain a criteria for the existence of first integrals for integrable codimension one deformations of quadratic  analytic  center-cylinder type singularities  in terms of the vanishing of some easy to compute line integrals.
\end{abstract}

\tableofcontents

\section{Introduction and main result}
\label{section:introduction} We study deformations of an unramified {\em complex analytic saddle}. This means, a singularity
given by the levels of the function germ $f=xy\colon (\mathbb C^2\times \mathbb C^{n-2}, 0) \to (\mathbb C,0)$ in coordinates $(x,y,z_1,...,z_{n-2})\in \mathbb C^2 \times \mathbb C^{n-2}$.
 For this sake we shall consider one-parameter analytic families of one-forms $\omega^t = d(xy) + \sum\limits_{j=1} ^\infty t^j \omega_j$ where $\omega_j$ is a germ of a holomorphic one-form at the origin $0\in \mathbb C^n$.
The one-form is assumed to be integrable, $\omega^t \wedge d \omega^t=0$ and depends analytically on the parameter $t\in \mathbb C_t$. Here $\mathbb C_t$ stands for the space of parameters $t$ and identifies with the space germ $(\mathbb C,0)$.
We shall study under which conditions the deformation admits a holomorphic first integral.

We  denote by $(f_c)$ the germ at $0\in \mathbb C^2$ of the level hypersurface $xy=c$, i.e., given by $\{(x,y,z)\in \mathbb C^2\times \mathbb C^{n-2}, \, xy=c\}$. By a {\it cycle} $\gamma_c \subset (f_c)$ we shall mean a continuous closed path in $(f_c)$. By $\oint _{\gamma_c} {\omega}^t$ we denote the line integral of the one-form $\omega^t$ along the closed path $\gamma_c\subset (f_c)$.

Our main result splits into two different cases according to the ambient dimension.

\begin{Theorem}
\label{Theorem:I}
Let $\{{\omega}^t\}_{t \in \mathbb C_t}$ be an analytic deformation of the complex saddle $df=0$ where $f=xy$ in
$(\mathbb C^2 \times \mathbb C^{n-2},0), \, n \geq 2$.  Then the following items are equivalent:

\begin{enumerate}[{\rm(i)}]
\item \, $\oint _{\gamma_c} {\omega}^t=0, \, \forall \gamma_c \subset (f_c), 0 \ne c \approx 0$.

\item \, $\omega^t=a^tdf + dh^t$ for some formal one-parameter
 families of holomorphic functions $a^t, h^t \in \mathcal O_n \otimes \hat{\mathcal O}_1$.
\end{enumerate}

For  $n\geq 3$, if the 2-form $\alpha:=\omega^t \wedge df$ has singular set of codimension $\ne 1$  at $0\in \mathbb C^n \times \mathbb C_t$ then the above conditions imply:

\begin{itemize}

\item[{\rm (iii)}] There is a {\em one-parameter holomorphic first integral} for $\omega^t$, i.e., there is a  holomorphic function $F \colon (\mathbb C^n\times \mathbb C_t, 0) \to (\mathbb C,0)$ such that
 each restriction $F^t(q)=F(q,t)$ is a holomorphic first integral for ${\omega}^t$ and $F^0=f$.

\end{itemize}

\end{Theorem}

Theorem~\ref{Theorem:I} can be seen as a version of the main result in \cite{scarduaIJM}, for the case
of non-dicritical singularities in dimension two. It is also connected to the original striking ideas and
spirit from \cite{Ilyashenko}.

In dimension two, we shall refer to the writing $\omega^t(x,y)=a^t(x,y)d (xy) + dh^t(x,y)$ as a
{\it standard  form} of the deformation. Notice that we have no control on the linear part
of the one-form $\omega^t$, i.e., it may not be a saddle.
This may kill any hope of using Martinet-Ramis (formal or analytic) normal forms for a further study of the above question.
Another point is that the above writing is not unique.
Nevertheless, we believe that in the case of dimension two either of the above conditions occurs then  there is a {\em one-parameter holomorphic first integral} for $\omega^t$.

Let us state a simplification of our result above. It is not difficult to see that, given any constant $c\in \mathbb C \setminus \{0\}$, the local 1-homology of the
level $(f_c) : xy=c$ is generated by a single closed path  $\tilde{\gamma}_c(t)=(x_o e ^{it}, c x_o^{-1} e^{-it}), \,
0 \leq t \leq 2 \pi$ where $x_o \neq 0$ is fixed and $i^2 =-1$.
In particular, we obtain as an immediate consequence of Theorem~\ref{Theorem:I}:

\begin{Corollary}
Given an analytic deformation of the complex saddle $d(xy)=0$ in
$(\mathbb C^2 \times \mathbb C^{n-2},0), \, n \geq 2$. Assume that $\sing({\omega}^t \wedge df)$ has codimension $\ne 1$ at $0\in \mathbb C^n \times \mathbb C_t$. Then $\omega^t$ admits a one-parameter holomorphic first integral  provided that we have $\oint_{\tilde{\gamma}_c} \omega^t= 0, \forall 0 \ne c \approx 0, \forall t\in \mathbb C_t$.

\end{Corollary}

\begin{Remark}
{\rm
Let us give a word about the meaning of the hypothesis {\em $\codim$ $\sing(\omega^t \wedge df)\ne 1$ at the origin $0\in \mathbb C^n \times \mathbb C_t$}. Assume that we are in the standard form case, ie., $\omega^t = a^t df + dh^t$ for some holomorphic germs
$a^t, h^t$ depending analytically on the parameter $t$. Then we have
$\omega^t \wedge df = dh^t \wedge df$. Let us assume that $n=3$. Hence  $f=xy$ in $\mathbb C^3$ with coordinates $(x,y,z)$. We have $\omega^t \wedge df = \omega^t \wedge d(xy) = dh^t \wedge d(xy)= (x(h^t)_x - y (h^t)_y)dx \wedge dy + y( h^t)_z dz \wedge dx + x(h^t)_z dz \wedge dy$. In this case the set
$Z:=\sing(\omega^t \wedge df) $ is given by the system of three equations
$x(h^t)_x - y (h^t)_y =0, \, x (h^t)_z=0, \, y (h^t)_z=0$. Therefore, either $h$ does not depend on $z$ and $\codim Z \leq 1$ or  we have $\codim Z \geq 2$. If $h=h(x,y)$ then $Z$ is given by the single equation
$x(h^t)_x - y (h^t)_y =0$ and therefore
either $Z$ is (contains) an open neighborhood of the origin or $\codim Z =1$. In the first case, by the local version of Stein factorization theorem (\cite{mattei-moussu} page 472 or \cite{Gr-Re})  we have $h^t(x,y) = \ell^t(xy)$ for some one variable holomorphic function
$\ell^t(z)$ depending analytically on the parameter $t$. Then we conclude that $\omega^t= a^t d(xy) + d \ell^t (xy)$ which implies that $xy$ is a first integral for $\omega^t$. In the second case $x(h^t)_x - y (h^t)_y$ is a nontrivial analytic germ. }
\end{Remark}

Our result above is sharp with respect to the condition $\codim \sing(\omega^t \wedge d(xy))\geq 2$ and
 on the vanishing of the integrals $\int_{\gamma_c} \omega^t$, as shown by Examples~\ref{Example:0}, \ref{Example:1} and \ref{Example:3}. Theorem~\ref{Theorem:I} is also not valid for deformations of the ramified saddle
 $x^n y^m =const$ as it  is shown by Example~\ref{Example:2}.

  Our results can be connected to a kind of a local version of Ilyashenko's celebrated result about limit cycles of two-dimensional  analytic ordinary differential equations (\cite{Ilyashenko}). Theorem~\ref{Theorem:I} also has common points with  the work of Mucino \cite{Mucino} about deformations of meromorphic pencils.
 It is  gives a test for the existence of first integrals in terms of computing line integrals which is quite practical.

\subsection{Real analytic case}

The second part of this work is dedicated to the problem of integrability of certain singularities related to real analytic centers.
The classical Lyapunov-Poincaré center theorem (\cite{Lyapunov,moussu,Poincare}) assures the existence of a first integral for an analytic 1-form near a center singularity in dimension two, provided that the first jet of the 1-form is nondegenerate.
We shall consider a real analytic differential 1-form $\omega(x,y)=a(x,y)dx + b(x,y)dy$ defined in a neighborhood of the
origin $0 \in \mathbb R^2$. Recall that the
1-form $\omega$ has a {\it center} at $0\in \mathbb{R}^2$ if all leaves in a punctured neighborhood of the
origin are diffeomorphic to the circle. The  form $\omega$ has a {\it real analytic first integral}
if $\omega=gdf$ for some real analytic function germs $f,g$ at $0\in\mathbb{R}^2$, with $g(0)\neq 0$. If moreover  $f$  has a Morse type singular point at the origin, then the form has a real analytic first integral {\it in the  strong sense}.

\begin{Theorem}[\cite{Lyapunov,Poincare}]
\label{Theorem:centertheorem} Consider a germ of a real analytic 1-form $\omega= a(x,y)dx+
b(x,y)dy$ at the origin $0\in \mathbb{R}^2$, having an isolated singularity for its first jet $\omega_1$, and a center at the origin. Then $\omega$ admits a first integral in the strong sense.
\end{Theorem}

In $\mathbb R^2 \times \mathbb R^{n-2}$ we consider coordinates $(x,y,z)=(x,y,z_1,...,z_{n-2})$.
We shall refer as {\it center-cylinder} singularity at the origin $0\in \mathbb R^2 \times \mathbb R^{n-2}$
to a product type singularity of a center singularity germ at $0\in \mathbb R^2$ by the axis $\mathbb R^{n-2}$.
This corresponds to a germ of a foliation with leaves diffeomorphic with the cylinder $S^1 \times \mathbb R^{n-2}$
and singular set consisting of the axis $\{0\}\times \mathbb R^{n-2}$.
The center-cylinder singularity is {\it quadratic} if
the center is given by the levels of a nondegenerate function in the real plane $\mathbb R^2$.

Given a germ of real analytic $q$-form $\alpha$ at $0\in \mathbb R^n$,
the {\it complex singular set} of $\alpha$ is defined as the
germ of complex analytic set at $0\in \mathbb C^n$ given by
the zeroes of the complexification of $\alpha$.
Following our main result Theorem~\ref{Theorem:I} we obtain the following application to this
real framework:

 For each $r>0$ denote by $\delta_r$ the circle $x^2 + y^2 = r^2$ in the real plane $\mathbb R^2$. Since any simple curve around the origin $0\in \mathbb R^2$ is, up to a change of orientation, homotopic to $\delta_r$ in $\mathbb R^2 \setminus \{0\}$, we have:

\begin{Corollary}
\label{Corollary:center} Let $\omega^t(x,y,z)= d(x^2 + y^2) + \sum\limits_{j=1}^\infty t^j \omega_j(x,y,z)$ be an analytic deformation of the
quadratic center-cylinder  singularity $d(x^2 + y^2) =0$ in $\mathbb R^2 \times \mathbb R^{n-2}, \, n \geq 2$ where $t \in (\mathbb R,0)$.
 Assume that $\oint_{\delta_r}\omega^t=0,\forall t, \forall r>0$ small enough. Then  we have $\omega^t(x,y,z)=a^t(x,y,z)d (x^2 + y^2) + dh^t(x,y,z)$ for some one-parameter families of real
analytic function germs $a^t, h^t$.

If moreover $(i) \, n \geq 3$ and $(ii)$ the complex singular set of $\omega^t \wedge d(x^2 + y^2)$ has codimension $\neq 1$ at the
  origin, 
then the deformation admits a one-parameter analytic first integral
$F^t(x,y,z) = x^2 + y^2 +  \sum\limits_{j=1}^\infty t^jF_j(x,y,z)$.
\end{Corollary}

We stress the fact the above statement does not assume that the 
linear part of the deformation is kept to be of the form 
$d(x^2 + y^2)$. 

\section{Analytic deformations of foliations}
\label{section:equations}

Before going into our results  we shall  explain notations and some concepts we use in the text.
We first recall that  the ring $\mathcal O_n$ of germs at $0\in \mathbb C^n$ of
holomorphic functions is an integral domain (the product of two elements of $\mathcal O_n$ distinct from the zero element can never be the zero element).  A unit is an element in the ring $\mathcal O_n$  having a multiplicative inverse in $\mathcal O_n$. This is precisely the set of germs of  functions which are nonzero at the origin. Since the nonunits form an ideal of $\mathcal O_n$ it
is a local ring. Furthermore, $\mathcal O_n$  is a unique factorization domain (cf. Theorem A.7 page 7 \cite{gunning2}).  A holomorphic function germ $f\in \mathcal O_n$ is called  {\it primitive}
if given a representative $f_D\colon D \to \mathbb C$ defined in a small disc $D\subset \mathbb C^n$
centered at the origin, then $f_D$ has connected fibers. In view of the local version of Stein factorization theorem (\cite{mattei-moussu} page 472 or \cite{Gr-Re}) this means that if we have another function $\tilde f \colon D \to \mathbb C$ such that $\tilde f$ is constant on each fiber $f_D^{-1}(c), c \in \mathbb C$ of $f_D$, then
$\tilde f$ can be  factorized as $\tilde f= \xi \circ f_D$ for some holomorphic map $\xi \colon U\subset \mathbb C \to V\subset \mathbb C$. Finally, in terms of germs, a germ $f \in \mathcal O_n$ is {\it primitive} if it is not a power, i.e., $f$ is not of the form $g^r$ for some $g \in \mathcal O_n$ and some $r \in \mathbb N, r \geq 2$ (see the factorization theorem in \cite{mattei-moussu} page 472).

\subsection{Deformations in dimension two}

 We start with the case of a germ a holomorphic one-form $\omega$ at $0\in \mathbb C^2$. In dimension two the integrability condition $\omega \wedge d \omega=0$ plays no role.
 A holomorphic function $f \in \mathcal O_2$ (respectively, a formal function $\hat f\in \hat {\mathcal O}_2$)  is called a {\it first integral} for $\omega$ if it is not constant and satisfies $df \wedge \omega=0$ (respectively, $d\hat  f \wedge \omega=0$). According to Malgrange and Mattei-Moussu (\cite{malgrangeI},\cite{mattei-moussu}),  if the origin is an isolated singularity for $\omega$, then the existence of a formal first integral is equivalent to the existence of a holomorphic one.

We consider a deformation  $\omega^{t}$ as in Theorem~\ref{Theorem:I}. Putting $\omega_0= df=d(xy)$
we have
\[
\omega^{t} = \omega_0 + \sum\limits_{j=1} ^\infty  t^j \omega_j\,.
\]
For each cycle $\gamma_c\subset (f_c)$ we have
\[
\int_{\gamma_c} \omega^t = \int_{\gamma_c} \omega_0
+ \sum \limits_{j=1}^\infty t^j \int _{\gamma_c} \omega_j.
\]
Therefore,  $\int_{\gamma_c}{\omega}^t=0, \forall t$ is equivalent to
$\int_{\gamma_c} \omega_j=0, \forall j\geq 0$.

Put now $\Omega^t:=\omega^t /f$ and $\Omega_j:=\omega_j /f$ where $f=xy$.
Then, because $f\equiv c$ in $\gamma_c$,  we have
\[
\int_{\gamma_c}{\omega}^t=0,\forall t \implies \int_{\gamma_c}\Omega_j =0, \forall j\geq 1.
\]

Now we need a lemma about relative cohomology (see also \cite{cerveau-berthier}):

\begin{Lemma}
\label{Lemma:II}
We consider a holomorphic one-form $\eta(x,y)$ germ at $0\in \mathbb C^2$. Then the following conditions are equivalent:
\begin{enumerate}[{\rm(i)}]
\item $\eta(x,y)$ writes as $\eta(x,y)= a(x,y) d(xy) + dh(x,y)$ for some holomorphic
functions $a(x,y), \, h(x,y)$ germs.

\item $\int_{\gamma_c} \eta =0, \forall \gamma_c \subset (f_c), \, \forall c \approx 0.$
\end{enumerate}

If we have an analytic  one-parameter  family $\{\eta^t\}_{t \in \mathbb C_t}$ such that $\int_{\gamma_c} \eta^t=0, \forall \gamma_c, \subset (f_c), \, \forall c \approx 0, \, \forall t \in \mathbb C_t $, then we can choose
$a^t, h^t$ as a formal family with respect to  the parameter $t$.

\end{Lemma}

\begin{proof}
Part $(i) \implies (ii)$ is straightforward. Assume now that we have (ii).
We write $\eta(x,y)= \sum\limits_{k,l \geq 0} A_{k,l} x^k y ^l dx +
\sum\limits_{k,l \geq 0} B_{k,l} x^k y ^l dy$ in convergent power series.
Choose now for each $ c \neq 0$ the path $\gamma_c(t)=(x_o e ^{it}, c x_o^{-1} e^{-it}), \,
0 \leq t \leq 2 \pi$ where $x_o \neq 0$ is fixed and $i^2 =-1$.

Then $\gamma_c \subset (f_c)$ as it is easily verified. We have
\[
\int_{\gamma_c} \eta = \sum\limits_{k,l \geq 0} A_{k,l} x_o^{k+1 - l} c^l \int \limits_{0}^{2\pi} e^{it(k+1 - l)}dt
+ \sum\limits_{k,l\geq 0} c^{l+1}(-i) x_o^{k +1 - l}B_{k,l} \int\limits_{0}^{2\pi} e^{it(k - l - 1)} dt.
\]

Notice that
\[
k+ 1 - l \ne 0 \implies \int\limits_{0}^{2 \pi} e^{it( k + 1 - l)}dt=0
\]
and
\[
k - l - 1 \ne 0 \implies \int \limits_{0}^{2\pi} e^{it(k - l - 1)} dt=0.
\]
Thus

\[
\int_{\gamma_c} \eta = \sum\limits_{l = k+1 \geq 1} A_{k,k+1} c^{k+1} \int\limits_{0}^{2\pi} 1 dt +
\sum\limits_{k = l + 1 \geq 1} (-i) c ^{k+1} B_{k, k -1} \int\limits_{0}^{2\pi} 1 dt.
\]

By hypothesis $\int_{\gamma_c} \eta=0, \forall c \ne 0$ so that $A_{k, k+1} = B_{k, k-1}, \forall k \geq 1$.

This allows us to formally solve the equation
\[
\eta \wedge d(xy) = dh(x,y) \wedge d(xy)
\]
obtaining a solution $h=\sum\limits_{k, l \geq 0} h_{k,l} x ^k y ^l$
which is in fact convergent.

Since $(\eta - dh) \wedge d(xy)=0$ we have by the classical division lemma that
\[
\eta - dh = a d(xy)\]
for some holomorphic function $a(x,y)$. The last part of the statement goes as follows: given an analytic family $\eta^t$ such that
$\int_{\gamma_c} \eta^t=0, \forall \gamma_c, \subset (f_c), \, \forall c \approx 0, \, \forall t \in \mathbb C_t $, we can write $\eta^t = \sum\limits_{j=0}^\infty t^j \eta_j$. Then we have $\int_{\gamma_c} \eta_j=0, \forall \gamma_c, \subset (f_c), \, \forall c \approx 0, \, \forall j\geq 0$. Using the above part we have $\eta_j = a_j d(xy) + dh_j$ and therefore
$\eta^t = \sum\limits_{j=0}^\infty (t^j a_j)d(xy) + \sum\limits_{j=0}^\infty t^jdh_j= \hat {a}^t d(xy) + d\hat {h}^t$ for obvious choices of $\hat {a}^t, {\hat h}^t \in \mathcal O_n \otimes \hat{\mathcal O_1}$.

\end{proof}

\section{Deformations of saddles in dimension $2$}
In this section we shall prove the first part of  Theorem~\ref{Theorem:I} for  $n=2$. We consider a deformation  $\omega^{t}=d(xy) + \sum\limits_{j=1}^\infty t^j \omega_j$ as in Theorem~\ref{Theorem:I}.
Assume that  we have
$\int_{\gamma_c}{\omega}^t=0, \forall \gamma_c, \forall c \approx 0$ and therefore
$\int_{\gamma_c}\omega_j=0, \forall j \geq 0, \forall \gamma_c$.
By Lemma~\ref{Lemma:II} above this implies
\[
\omega_j = a_j(x,y) d(xy) + dh_j(x,y)
\]
for some holomorphic functions $a_j,\, h_j$.
This gives
\[
{\omega}^t = d(xy) + \sum\limits_{j=1}^\infty t^j (a_j d(xy) + dh_j)
= (1+ \sum\limits_{j=1}^\infty t^j a_j) d(xy) + \sum\limits_{j=1}^\infty t^j dh_j.
\]
We may then write ${\omega}^t = a^t d(xy) + dh^t$ for obvious choices of $a^t$ and $h^t$ which are elements of $\mathcal O_n \otimes \hat{\mathcal O_1}$. This first pair of choices gives functions $a^t, h^t$ that may not depend analytically on $t$ in a convergent way, they may be only formal with respect to this variable. Thus we have $(i)\implies (ii)$ in Theorem~\ref{Theorem:I} for $n=2$.
Let us now assume that
$\omega^t= a^t d(xy) + dh^t$ as stated. For each fixed value $t$ of the parameter we have
\[
\int_{\gamma_c} \omega^t = \int_{\gamma_c} (a^t d(xy) + dh^t)=
\int_{\gamma_c} a^t .0 + \int_{\gamma_c} dh^t = 0.
\]
This proves that $(ii)\implies (i)$ in Theorem~\ref{Theorem:I}.

\section{Extension theorems}
Before going into the proof of Theorem~\ref{Theorem:I} for the case $n\geq 3$, we shall gather a few elements, mostly from
\cite{mattei-moussu}.
Let $\Omega$ be a germ of an integrable 1-form at $0\in \mathbb C^n, n \geq 3$.
We shall assume that $\codim\sing(\Omega)\geq 2$.
We shall say that a holomorphic  embedding $\mathcal I \colon (\mathbb C^2,0) \to (\mathbb C^n,0)$ is
{\it in general position with respect to} $\omega$ if:
\begin{enumerate}[{\rm(1)}]
\item $\sing(\mathcal I ^*(\Omega) )=\mathcal I^{-1}(\sing(\Omega))$.

\item $\codim\sing(\mathcal I^*(\Omega))=\inf\{\codim \sing(\Omega), 2\}$.
\end{enumerate}

According to the transversality theorem in \cite{mattei-moussu} (Theorem 2 page 508) there is always
a holomorphic embedding in general position with respect to $\Omega$.

We shall need the following result from \cite{cerveau-berthier} (cf. Extension theorem page 405):

\begin{Lemma}
\label{Lemma:extensionlemma}
Let $\omega, \eta$ be germs of holomorphic one-forms at $0\in \mathbb C^n, \, n \geq 3$. Assume that $\omega$ is integrable, $\omega \wedge d \omega=0$ and $\eta$ is closed in the leaves of $\omega$, \, $d \eta \wedge \omega=0$.
Let $ i \colon (\mathbb C^p,0) \to (\mathbb C^n,0)$ be a holomorphic embedding in general position with respect to $\omega$ where $ 2 \leq p \leq n$.
If we have $i^* (\eta) = a_0 i^* (\omega) + d h_0$ for some germs $a_0, h_0 \in \mathcal O_p$ then there are germs $a, h \in \mathcal O_n$ such that
$a_0 = i^* a, \, h_0 = i^* h$ and $\eta = a \omega + dh$.
The germs $a, h$ above are unique with the properties listed. 

There is a parametric version: if $\eta^t$ is an analytic family with $t\in \mathbb C_t$, $d\eta^t \wedge \omega=0$,  and if we can write $i^*(\eta^t)=a_0 ^t i^*(\omega) + dh_0 ^t$ for some one-parameter functions $a_0 ^t, h_0^t \in \mathcal O_n \otimes \hat{\mathcal O}_1$, then we can write $\eta_ t = a^t \omega + d h^t$ for some one-parameter functions $a^t, h^t \in \mathcal O_n \otimes \hat{\mathcal O_1}$.

\end{Lemma}
Lemma~\ref{Lemma:extensionlemma} is found in \cite{cerveau-berthier} (cf. Extension theorem page 405) in the non-parametric form. Nevertheless, the parametric version can be obtained with the same proof.

From the above lemma we obtain:

\begin{Lemma}
\label{Lemma:corollary}
Let $\{\eta^t\}_{t \in \mathbb C_t}$ be an analytic family of  germs of holomorphic one-forms at $0\in \mathbb C^n, n \geq 2$ and assume that:
\begin{enumerate}
\item $d\eta^t \wedge d(xy)=0, \forall t \in \mathbb C_t$.
\item  $\oint_{\gamma_c} \eta^t =0, \forall \gamma_c \subset (f_c), \, \forall c \approx 0, \forall t \in \mathbb C_t.$
\end{enumerate}
Then we have $ \eta^t = a^t d(xy) + dh^t$ for some one-parameter formal family of germs $a^t, h^t \in \mathcal O_n\otimes \hat{\mathcal O_1}$.

\end{Lemma}

\begin{proof}
We choose a holomorphic embedding $i \colon (\mathbb C^2, 0) \to (\mathbb C^n,0)$ in general position with respect to $d(xy)$.    Let us put $(\tilde x, \tilde y):=(i^* x, i^* y)$ local coordinates at  $0 \in \mathbb C^2$.
The restriction $\eta^t_0 = i^* \eta^t$ satisfies $\oint_{\gamma_c} \eta^t_0=0, \forall \gamma_c \subset (\tilde x \tilde y=c)$. Therefore, from Lemma~\ref{Lemma:II} there are germs $a^t_0, h^t_0 \in \mathcal O_2 $ such that $\eta^t_0 = a^t _0 d(\tilde x \tilde y) + dh^t_0$. By Lemma~\ref{Lemma:extensionlemma} we have $\eta^t = a^t d(xy) + dh^t$ for some formal family of holomorphic  germs $a^t, h^t \in \mathcal O_n\otimes \hat{\mathcal O_1}$.
\end{proof}

Another important tool is the following extension theorem:

\begin{Lemma}[\cite{mattei-moussu} Theorem~1 page 507]
\label{Lemma:extension}
Let $\Omega$ be a germ of an integrable 1-form at the origin $0 \in \mathbb C^n$ and
let $\mathcal I \colon (\mathbb C^2,0) \to (\mathbb C^n,0)$ be a holomorphic embedding
in general position with respect to $\Omega$. If $f_o\colon (\mathbb C^2,0) \to (\mathbb C,0)$ is
a germ of a holomorphic first integral for $\omega=0$ then there is a unique function
$f\colon  (\mathbb C^n,0)\to (\mathbb C,0)$ which is a first integral for $\Omega$ and such that
$\mathcal I^*(f)=f_o$, i.e., $f_o=f \circ  \mathcal I$.
\end{Lemma}
Similarly to Lemma~\ref{Lemma:extensionlemma}, Lemma~\ref{Lemma:extension} above admits a parametric version obtained with similar arguments to the ones used for the non-parametric version. 
Indeed, the extension theorem above holds with the same proof for one-parameter first integrals of
analytic deformations. It is only enough to observe that given such a deformation $\Omega^t$ in $(\mathbb C^n,0)$,
if a holomorphic embedding $\mathcal I \colon (\mathbb C^2,0) \to (\mathbb C^n,0)$ is in general position with
respect to $\Omega_o$ then the same holds for $\Omega^t$ for $t$ close enough to $0$.

\section{Deformations of saddles in dimension $n \geq 3$}

In this section we shall finish the proof of  Theorem~\ref{Theorem:I}.
We consider an analytic deformation $\{{\omega}^t\}_{t \in \mathbb C,0}$  of the complex saddle $d(xy)=0$ in
$(\mathbb C^2 \times \mathbb C^{n-2},0), \, n \geq 3$. We write $\omega^t = d(xy) + t\omega_1 + t^2 \omega_2 + \cdots$. We  assume that $\int_{\gamma_c}{\omega}^t=0, \forall \gamma_c, \forall c \approx 0$.
Our first  step is:

\begin{Lemma}
\label{Lemma:coefficientforms}
There are formal families of holomorphic function germs $a^t, h^t \in \mathcal O_n \otimes \hat{\mathcal O_1}$ such that for each $t$ we have $\omega^t = a^t d(xy) + dh^t$. 
\end{Lemma}
\begin{proof}
Let us write $\omega^t = d(xy) + t \alpha$. Then $d \omega^t =
t d\alpha$. From the integrability condition $\omega^t \wedge d\omega^t=0$ we have $0=td(xy)\wedge d \alpha + t^2 \alpha \wedge d \alpha$.
Therefore we must have $0 = d(xy) \wedge d \alpha = \alpha \wedge d \alpha$.  Hence $d\omega^t \wedge d(xy) = t d\alpha\wedge d(xy)=0$.
 We also have by hypothesis $\oint_{\gamma_c} \omega^t = 0, \forall \gamma_c \subset (xy=c)$.

  Using now Lemma~\ref{Lemma:extensionlemma} we conclude that $\omega^t = a^t d(xy) + dh^t$ for some formal families of holomorphic germs $a^t, h^t \in \mathcal O_n\otimes \hat{\mathcal O_1}$. 
\end{proof}

We shall need another lemma:

\begin{Lemma}
\label{Lemma:III} If $\codim \sing(\omega^t \wedge d(xy))\geq 2$ then there is a holomorphic function $F(x,y,z,t)$ such that for each $t$, the restriction
$F^t(x,y,z)=F(x,y,z,t)$ is a first integral for ${\omega}^t$.

\end{Lemma}

\begin{proof}
Let us write $\hat A(x,y,z,t)=a^t(x,y,z), \,  \hat H(x,y,z,t)=h^t(x,y,z)$. 
We have ${\omega}^t(x,y,z) = \hat A(x,y,z,t) df + d_{x,y,z}\hat H(x,y,z,t)$ where
$\hat A(x,y,z,t) = 1 + \sum\limits_{j=1}^\infty a_j(x,y,z) t^j, \,
\hat H(x,y,z,t) = \sum\limits_{j=1}^\infty h_j(x,y,z) t^j$ and $f=xy$.
The functions $a_j(x,y,z),\,  h_j(x,y,z)$ are holomorphic.
The notation is clear: \[ d_{x,y,z} \hat H(x,y,z,t)=\sum\limits_{j=1}^\infty t^j dh_j(x,y,z)\] and so on. Let us now put
$\Omega(x,y,z,t):=\omega_t(x,y,z)$ seen as a holomorphic one-form.
This one-form is not necessarily integrable, indeed $d \Omega(x,y,t) =
d_{x,y,z} {\omega}^t (x,y,z) + \frac{\partial {\omega}^t}{\partial t}\wedge dt$ in the same natural notation
above.
Hence
$\Omega \wedge d \Omega = {\omega}^t \wedge \frac{\partial {\omega}^t}{\partial t} \wedge dt$.

Let $\mathcal S$ be the system of holomorphic one-form germs generated by $\Omega$ and $dt$, {\it i.e.},
$\mathcal S = \mathcal S(\{\Omega, dt\})$ regarded as a module over $\mathcal O_n\{\{t\}\}$.

\begin{Claim}
Regarding the above system $\mathcal S$ we have:
\begin{enumerate}[{\rm(i)}]
\item $\mathcal S$ is integrable.
\item $\mathcal S= \mathcal S\{\hat A df + d \hat H, dt\}$.
\end{enumerate}
\end{Claim}

\begin{proof}$\;$
	
(i) Indeed, $\Omega \wedge d \Omega \wedge dt  = {\omega}^t \wedge \frac{\partial {\omega}^t}{\partial t} \wedge dt \wedge
		dt =0$.
		
(ii) We have $d\hat H = d_{x,y,z} \hat H + \frac{\partial d \hat H}{\partial t} dt$ where
		$\frac{\partial \hat H}{\partial t} dt = \sum\limits_{j=1}^\infty j t^{j-1} h_j(x,y,z) dt$.
		Since $\sum\limits_{j=1}^\infty j t^{j-1} h_j(x,y,z) \in \mathcal O_n\{\{t\}\}$ we obtain the result stated. 
\end{proof}

Put now $\hat \alpha (x,y,z,t):= \hat A df + d \hat H$. We claim:
\begin{Claim}
We have $\hat \alpha \wedge d \hat \alpha \wedge dt =0$.

\end{Claim}
\begin{proof}
Indeed, since the system $\mathcal S$ is integrable this implies that $\hat \alpha \wedge d \hat \alpha \wedge dt=0$.
\end{proof}

Therefore we have $d \hat H \wedge d \hat A \wedge df \wedge dt =0$. From the expression for $d \hat H$ we then obtain
$d_{x,y,z} \hat H \wedge d_{x,y,z} \hat A \wedge df =0$ and therefore 
\[
d_{x,y,z} \hat A \wedge (d_{x,y,z} \hat H \wedge  df) =0.
\]

Now we observe that ${\omega}^t \wedge df = (\hat A  df + d_{x,y,z} \hat H ) \wedge df = d_{x,y,z} \hat H \wedge df$.
Therefore, if we assume that $\codim _{\mathbb C^3}\sing({\omega}^t \wedge df) \geq 2$ then we have 
\[
\codim _{\mathbb C^3}\sing(d_{x,y,z} \hat H \wedge df)\geq 2.
 \]
 
 By a parametric formal version of Stein factorization theorem (\cite{Hartshorne}, \cite{malgrangeII}) we conclude that there is a formal function
$\hat \vr (w_1, w_2, t)= \sum\limits_{j=0}^\infty t^j \vr_j(w_1,w_2)$ such that
$\hat A(x,y,z,t)=\hat \vr(f(x,y),\hat H(x,y,z,t),t)$ or equivalently 

\[
\hat A^t(x,y,z)= \hat\vr^t (f(x,y),\hat H^t(x,y,z))
 \]
in the obvious sense. Thus we can write
\[
\hat \alpha = \hat A df + d \hat H = \hat \vr (f, \hat H, t) df + d \hat H.
\]
Therefore
$\mathcal S = \mathcal S \{ \hat \vr (f, \hat H , t) df + d \hat H, dt\}$.
The system $\mathcal S$ is therefore the formal pull-back of the non-singular system
$\mathcal S_o = \mathcal S \{ \hat \vr (w_1, w_2, t) dw_1 + dw_2, dt\}$ by the formal map
$\hat \sigma = (f, \hat H, t)$. Since $\mathcal S_o$ is non-singular (though it is formal)
it admits a formal first integral (formal Frobenius theorem) given by a pair
$(\hat F_o (w_1, w_2, t), t)$. The same holds for $\mathcal S$, i.e., $\mathcal S$ admits a formal
first integral of the form $(\hat F_o(f, \hat H, t), dt)$. Nevertheless, the system
$\mathcal S$ is {\em holomorphic} and with singular set of codimension $\geq 2$. Thus,
by Malgrange \cite{malgrangeII} there is a holomorphic first integral for $\mathcal S$ which can be
chosen of the form $(F(x,y,z,t), t)$. The holomorphic function $F(x,y,z,t)$ is then such that
$F^t$ is a holomorphic first integral for  $\omega$ for each value of the parameter  $t$.

\end{proof}

\begin{proof}[End of the proof of Theorem~\ref{Theorem:I}]
There is very few remaining. By hypothesis $\codim \sing(\omega^t \wedge d(xy))\ne 1$. If $\codim \sing(\omega^t \wedge d(xy))=0$ then we $\omega^t \wedge d(xy))=0, \forall t$ and $f=xy$ is a first integral for $\omega^t, \forall t \in \mathbb C_t$, the deformation is trivial. 
Assume now that $\codim\sing (\omega^t \wedge d(fg))\geq 2$. Then we may apply above lemmas and conclude that $\omega^t$ admits a one-parameter first integral $F^t$. 
\end{proof}

\section{Examples}
This section contains examples proving that the result in Theorem~\ref{Theorem:I} is sharp. The first couple of
examples show that the $\sing({\omega}^t \wedge df)$ has codimension $\geq 2$ condition cannot be dropped.
\begin{Example}
\label{Example:0}
{\rm
We consider the family $\{{\omega}^t\}_{t\in \mathbb C}$ of germs of one-forms at $0\in \mathbb C^2 \times \mathbb C^{n-2}$, given by ${\omega}^t= d(xy) + t(xy)^2dx$ in coordinates
$(x,y,z_1,...,z_{n-2})\in \mathbb C^2\times \mathbb C^{n-2}$. Then $\omega_0=d(xy)$ and
we have

(i) $\dint_{\gamma_c}{\omega}^t = \int_{\gamma_c} (d(xy) + t(xy)^2dx) = t . c^2. \int_{\gamma_c} dx=0, \forall \gamma_c
$
where $\gamma_c$ runs all over the cycles in $(f_c): xy=c\ne 0$.

Nevertheless

(ii) ${\omega}^t \wedge d(xy) = t(xy)^2 dx \wedge d(xy) = t(xy)^2 x dx \wedge dy = tx^3 y^2 dx \wedge dy
$.

Thus $\sing({\omega}^t \wedge d(xy)) = (tx^3 y^2=0)$ and for $t\ne 0$ we have
$\sing({\omega}^t \wedge d(xy)) = (x=0) \cup (y=0)$ which violates condition
$\codim_{\mathbb C^3,0}({\omega}^t\wedge df) \geq 2$.

Finally, we observe that:

\begin{Claim}
The one-form ${\omega}^t$ does not admit a holomorphic first integral for $t \ne 0$.
\end{Claim}
Indeed, for each value $t\in \mathbb C$ fixed we have ${\omega}^t / (xy)^2  = d(xy)/(xy)^2 + tdx =
d( - (xy)^{-1}) + t dx = d \big( - (xy) ^{-1}  + tx\big)$. The existence of a pure meromorphic first integral, for $t\ne 0$,  shows that this foliation admits no holomorphic first integral (see \cite{Cerveau-Mattei} Chapter~II, Theorem 1.1,
page 106, for a study of
foliations with pure meromorphic first integrals).

}
\end{Example}

A more elaborate version of the above example is given below:

\begin{Example}
\label{Example:1}
{\rm
We consider the family $\{{\omega}^t\}$ defined by ${\omega}^t=e^{ty} d(xy) + txydx$ in coordinates
$(x,y,z_1,...,z_{n-2})\in \mathbb C^2\times \mathbb C^{n-2}$. Then $\omega_0=d(xy)$ and
\[
{\omega}^t= d(xy) + t(y^2 d(xy) + xydx) + t^2 y^4/2 d(xy) + \ldots
\]

We have

(i) $\dint_{\gamma_c}{\omega}^t = \int_{\gamma_c} (e^{ty^2} d(xy) + txy)dx = t . c. \int_{\gamma_c} dx=0, \forall \gamma_c
$
where $\gamma_c$ runs all over the cycles in $(f_c): xy=c\ne 0$.

Nevertheless

(ii) ${\omega}^t \wedge d(xy) = txy dx \wedge d(xy) = txy x dx \wedge dy = tx^2 y dx \wedge dy.
$

Thus $\sing({\omega}^t \wedge d(xy)) = (tx^2 y=0)$ and for $t\ne 0$ we have
$\sing({\omega}^t \wedge d(xy)) = (x=0) \cup (y=0)$ which violates condition
$\codim_{\mathbb C^3,0}({\omega}^t\wedge df) \geq 2$.

Finally, we observe that:

\begin{Claim}
The one-form ${\omega}^t$ does not admit a holomorphic first integral for $t \ne 0$.

\end{Claim}

The proof is based on the computation of the first terms of the holonomy map
of the separatrix $(y=0)$ and in showing that this map is not finite (indeed, if such a map is finite then since it
is tangent to the identity, it must be trivial, which is not the case). According then to \cite{mattei-moussu}
this implies that there is not holomorphic first integral for the foliation.

}
\end{Example}

\begin{Example}
\label{Example:2}
{\rm
This next example shows that the result is not valid for saddles of the form $f=x^n y^m=const$ with $(n,m)\ne (1,1)$ at $0\in \mathbb C^2 \times \mathbb C^{n-2}$. Indeed,
consider the one-form $\omega_0=xy d(x^2 y^3)/(x^2 y^3)= 2ydx + 3 xdy$. This one-form corresponds to the saddle $x^2 y^3 = c\in \mathbb C$. We can consider the deformation ${\omega}^t : = \omega_0 + t \omega_1$ given by
$\omega_1 =  a_1 \omega_0 + dh_1$ where $a_1$ is a constant and $h_1=xy$.
Then we have ${\omega}^t = \omega_0 + ta_1 \omega_0 + t dh_1$ and therefore
\[
{\omega}^t/(xy)=(1+ t a_1) d(x^2 y^3)/(x^2 y^3) + t d(xy)/(xy).
\]

This shows that, for $a_1\in\mathbb C$ the one-form ${\omega}^t$ admits a  first integral of liouvillian type
given by $f=(x^2 y^3) ^{1 + a_1 t} (xy)^t$ which is not holomorphic in general. This one-form shall not
admit a holomorphic first integral.

}
\end{Example}

\begin{Example}
\label{Example:3}
{\rm We shall now give an example of a deformation of the saddle $d(xy)=0$ for which there is
not holomorphic first integral. The reason will be that the deformation fails to meet the integral condition
$\oint _{\gamma_c} \omega^t =0$ along the cycles $\gamma_c \subset (f_c)$.
For this we consider a one-form
\[
\omega_{k,\lambda}=(1+ (\lambda -1)xy)ydx + (1+ \lambda xy) xdy + g(z)dz
\]
where $g(z)$ is a holomorphic germ at the origin with $g(0)=0$.
This describes a perturbation of the  classical formal models due to Martinet-Ramins (\cite{martinet-ramisresonant}) (obtained when $g=0$) for the resonant singularities $d(xy) + \Omega_2(x,y)=0$ where $\Omega_2$ has order
$\geq 2$ at the origin.  We can rewrite
$\omega_{k,\lambda}=(1+ \lambda xy)d(xy) - xy^2 dx$.

If we put
\[
\omega^t=(1+ t \lambda xy)d(xy) - t xy^2 dx + tg(z)dz
\]
then we obtain an analytic deformation  of $\omega^0=\omega_0:=d(xy)$.
For a  cycle $\gamma_c\subset (xy=c), \, c \ne 0$ we have
\[
\int_{\gamma_c} \omega^t = (1+ t \lambda c) \int_{\gamma_c} d(xy)  -
t c \int _{\gamma_c} ydx = -  t c \int _{\gamma_c} ydx \ne 0, \forall t \ne 0.
\]

The fact that the last integral is not zero is an easy consequence of Green-Stokes theorem, or
via direct computation.
Finally, the formal model and therefore $\omega^t$ for $t\ne 0$,  does not admit
a holomorphic first integral as it is well-known since the work of Martinet-Ramis mentioned above.

}
\end{Example}

\section{Real analytic centers}

In this section we shall prove Corollary~\ref{Corollary:center}.
We shall need the following version of Lemma~\ref{Lemma:III}:

\begin{Lemma}
\label{Lemma:center}
Let  $\eta(x,y)= A(x,y)dx + B(x,y)dy$ be a germ of analytic 1-form at $0\in \mathbb R^2$ such that $\oint_{\delta_r} \eta =0, \forall r> 0$ small enough. Then  $\eta(x,y)$ writes as $\eta(x,y)= a(x,y) d(x^2 + y^2) + dh(x,y)$ for some germs of analytic
functions $a(x,y), \, h(x,y)$. Moreover, there is a parametric version:
if $\eta^t$ depends analytically on the parameter $t\in \mathbb R_t$ and 
 $\oint_{\delta_r} \eta^t =0, \forall r> 0$ small enough, $\forall t \in \mathbb R_t$,  then we may choose $a^t$ and $h^t$ such that $\eta^t = a^t d(x^2 + y^2) + dh^t$ and $a^t, h^t$ being formal families of real analytic function germs.

\end{Lemma}

\begin{proof}
We start with the non-parametric version. 
We introduce the complex variables $z=x + iy$ and $w=x- iy$ where $i^2 =-1$.
Then we have $x^2 + y^2 = zw$. The complexification of $\eta$ is a holomorphic
1-form germ $\eta_\mathbb C$ at the origin $0\in \mathbb C^2$. The curve $\delta_r$ induces a cycle
$\gamma_{r^2}$ in the leaf $(f_c): zw =c$ of the saddle $df=0$ where $f=zw$. This cycle is nontrivial and therefore
it is a generator of the 1-homology of $(f_c)$. Therefore $\oint_{\delta_r} \eta =0 \implies \int _{\gamma_{r^2}}\eta=0$
for all closed curve in $(f_c)$ for the values $c=r^2$. Nevertheless, it is not clear, a priori, that the integral above vanishes for all the values $0 \ne c\in \mathbb C$. The alternative is then to reproduce the proof of Lemma~\ref{Lemma:II}. In order to see this, we observe that by a suitable choice of the generator $\gamma_c\subset (f_c)$ of the 1-homology of the leaf $(f_c): zw=c$ we may consider the map $I(\eta)\colon c \mapsto \int_{\gamma_c} \eta$ as a holomorphic map on the parameter $c\in D\setminus \{0\}$ where $D\subset\mathbb C$ is a small disc centered at the origin $0\in \mathbb C$. Indeed, this is easy to see from the fact that we may choose paths $\gamma_c$ based on the
points $(z_0,w_0)=(z_o,c/z_o)$ for a choice of $z_o\ne 0$ close enough to $0$. This gives then a holomorphic function
$I(\eta)\colon D\setminus \{0\} \to \mathbb C$ which satisfies $I(\eta)(r^2)=0, \forall r>0$ small enough. The identity principle then shows that $I(\eta)$ is identically zero.
Another possibility is to go back in the proof of Lemma~\ref{Lemma:II}  and observe that
the hypothesis on the real curves $\delta_r$ implies that, using the same notation from the proof of
Lemma~\ref{Lemma:II}, we must have
\[
\sum\limits_{l = k+1 \geq 1} A_{k,k+1} (r^2)^{k+1} \int\limits_{0}^{2\pi} 1 dt +
\sum\limits_{k = l + 1 \geq 1} (-i) (r^2) ^{k+1} B_{k, k -1} \int\limits_{0}^{2\pi} 1 dt = 0
\]
for all $r>0$ small enough. This already enough to assure that
$A_{k, k+1} = B_{k, k-1}, \forall k \geq 1$. Then we can proceed as in the proof of Lemma~\ref{Lemma:II} to conclude
that we have $\eta_\mathbb C (z,w) = A d(zw)+ dH(z,w)$ for some holomorphic functions $a,h$.
Since $\eta_\mathbb C$ is the complexification of $\eta$ and $z=x+ iy, w = x - iy$, we obtain real analytic functions $a(x,y), h(x,y)$ such that $\eta = a d(x^2 + y^2) + dh$. As for the parametric version we write 
$\eta^t = \sum\limits_{j=0} ^\infty t^j\eta_j$. Then $\oint_{\delta_r} \eta =0, \forall t \implies \oint_{\delta_r} \eta_j =0, \forall j\geq 0$. Using the above we may write $\eta_j = a_j d(x^2 + y^2) + dh_j$ for some analytic function germs $a_j, h_j$. Then we obtain the writing $\eta^t = \sum\limits_{j=0}^\infty t^j(a_j d(x^2 + y^2) + dh_j) = 
(\sum\limits_{j=0}^\infty t^j a_j) d(x^2 + y^2) + \sum\limits_{j=0}^\infty 
(t^jdh_j)= a^t d(x^2 + y^2) + dh^t$ for obvious choices of $a^t, h^t$.

\end{proof}

Using the above result as well as the extension techniques used in the proof of Lemma~\ref{Lemma:corollary} we obtain:

\begin{Lemma}
\label{Lemma:corollarycenter}
Let $\{\eta^t\}_{t \in \mathbb C_t}$ be an analytic family of  germs of real analytic one-forms at $0\in \mathbb R^n, n \geq 2$ and assume that:
\begin{enumerate}
\item $d\eta^t \wedge d(x^2 + y^2)=0, \forall t \in \mathbb R_t$
\item  $\oint_{\delta_r} \eta^t =0, \forall t \in \mathbb R_t$ and $\forall r>0$ small enough.
\end{enumerate}
Then we have $ \eta^t = a^t d(x^2 + y^2) + dh^t$ for some one-parameter formal family of real analytic function germs $a^t, h^t$ at $0\in \mathbb R^n$.

\end{Lemma}

We are now in condition to prove our second main result:

\begin{proof}[Proof of Corollary~\ref{Corollary:center}]
The first part is a consequence of Lemma~\ref{Lemma:center}.
By hypothesis we have $\oint_{\delta_r}\omega^t=0$ for all $r>0$ small enough, where
$\delta_r$ denotes the circle $x^2 + y^2 = r^2$ in $\mathbb R^2$. From the hypothesis
$\oint_{\delta_r}\omega^t=0, \forall t, \forall r>0$ small enough, we conclude via Lemma~\ref{Lemma:center} that
we may write each $\omega_j= a_j d(x^2+ y^2) + dh_j$ for some real analytic functions
$a_j, h_j$.
We consider the complexification $\omega^t _\mathbb C$  of the
deformation $\omega^t = d(x^2 + y^2) + \sum\limits_{j=1}^\infty t^j \omega_j(x,y)$.

Then, in coordinates $z=x+ iy, w = x -iy$ we have the complexification given by
$\omega_\mathbb C ^t = d(zw) + \sum\limits_{j=1}^\infty t^j A_j d(zw) + dH_j$ where $A_j$ and
$H_j$ are the holomorphic functions obtained by the complexification of $a_j$ and $h_j$ respectively.
Now we observe that so far $\omega_\mathbb C^t$ is a one real parameter analytic deformation, because
$t$ is a real parameter. Now we introduce the complex parameter $T\in \mathbb C$ having $\Re T =t$, say $T =
t + is$. Then we can define the one complex parameter analytic deformation $\Omega^ T:=
d(zw) + \sum\limits_{j=1} ^\infty T^j (A_jd(zw) + dH_j)$. Then clearly from this expression we have
$\int_{\gamma_c} \Omega^T=0$ for all cycle $\gamma_c\subset (zw=c)$ for all $c\in \mathbb C\setminus \{0\}$ close enough to zero.
Finally, we observe that by the hypothesis the singular set  $\sing(\Omega^T\wedge d(zw))$ has codimension $\geq 2$ at the origin of
  $\mathbb C^n \times \mathbb C_T$. This implies by our Theorem~\ref{Theorem:I} that $\Omega^T$ admits a one-parameter
  holomorphic first integral say $F=F^T\colon (\mathbb C^n \times \mathbb C_T, 0)\to (\mathbb C,0).$ In particular
  the decomplexification $\omega^t$ admits a real analytic one-parameter first integral $f=f^t\colon (\mathbb R^n \times \mathbb R_t,0) \to (\mathbb R,0)$. This ends the proof of Corollary~\ref{Corollary:center}.

\end{proof}

\end{document}